\documentclass{amsproc} 
\usepackage{euscript}
\usepackage{cases}
\usepackage{mathrsfs}
\usepackage{bbm}
\usepackage{amssymb}
\usepackage{amsfonts,amsmath,amsxtra,mathdots,mathabx}
\usepackage{color}
\usepackage{hyperref}
\usepackage{tikz}
\usepackage{appendix,upgreek,setspace}

\allowdisplaybreaks

\DeclareFontFamily{U}{matha}{\hyphenchar\font45}
\DeclareFontShape{U}{matha}{m}{n}{
	<5> <6> <7> <8> <9> <10> gen * matha
	<10.95> matha10 <12> <14.4> <17.28> <20.74> <24.88> matha12
}{}
\DeclareSymbolFont{matha}{U}{matha}{m}{n}

\DeclareMathSymbol{\Lt}{3}{matha}{"CE}
\DeclareMathSymbol{\Gt}{3}{matha}{"CF}

\DeclareSymbolFont{mathc}{OML}{txmi}{m}{it}% txfonts
\DeclareMathSymbol{\varuu}{\mathord}{mathc}{117}
\DeclareMathSymbol{\varvv}{\mathord}{mathc}{118}
\DeclareMathSymbol{\varww}{\mathord}{mathc}{119}

\def\valpha{\text{\scalebox{0.84}[1.02]{$\alpha$}}}   
\def\vepsilon{\upvarepsilon}
\def\vnu{\text{{\scalebox{0.9}[1]{$\nu$}}}} 
\def\vkappa{\text{{\scalebox{0.86}[1.1]{$\kappa$}}}}

\def\bfF{\mathbf{F}}
\def\bfQ{\mathbf{Q}}

\def\bfH{\mathbf{H}}

\newcommand{\As}{{\mathrm {As}}}

\def\SO {\text{\raisebox{- 2 \depth}{\scalebox{1.1}{$ \text{\usefont{U}{BOONDOX-calo}{m}{n}O}  $}}}}
\def\RU {\mathrm{U}}

\newcommand{\BQ}{{\mathbf {Q}}} 
\newcommand{\BR}{{\mathbf {R}}} 
\newcommand{\BZ}{{\mathbf {Z}}}

\def\Tr{ \mathit{Tr}}

\newcommand{\SL}{{\mathrm {SL}}}

\newcommand{\ra}{\rightarrow} 
\def\sumx{\sideset{}{^\star}\sum}

\def\nd{\mathrm{d}}
\def\ro{\mathrm{o}}

\def\lp {\left(}
\def\rp {\right)}

\newcommand{\delete}[1]{}

\theoremstyle{plain}

\newtheorem{coro}{Corollary}[section]
\newtheorem{lem}{Lemma}[section]
\newtheorem{theorem}{Theorem}[section] 
\newtheorem{proposition}{Proposition}[section]

\newtheorem*{thm*}{Theorem}

\theoremstyle{remark} 
\newtheorem{remark}{Remark}[section] 
\newtheorem{defn}{Definition}[section]

\numberwithin{equation}{section}

\begin{document}
	
	\title{On the Second Moment of $L (1/2, \As (f))$}
	
	\begin{abstract}
		Let $\mathbf{F}$ be a real quadratic field. 
	 Let $f $ traverse a Hecke orthonormal basis of Hilbert cusp forms over $ \mathbf{F} $ of full level and parallel weight $(k,k)$. As $k \ra \infty$, 
		we prove an asymptotic formula   for the second moment of central Asai $L$-values $L (1/2, \As (f))$:   
		\begin{equation*}
			{\sum}_{f } \, \omega_f L(1/2,\As(f))^2 =   P_3 ( \log  {k  }   )  k^2 + O_{\mathbf{F},\upvarepsilon} (k^{3/2 + \upvarepsilon} ),
		\end{equation*} 
	where $\omega_f$ are the harmonic weights and $P_3 (X)$ is an explicit polynomial of degree $3$. 
	This refines the mean Lindel\"of bound $ O_{\mathbf{F},\upvarepsilon} (k^{2 + \upvarepsilon} ) $ proved by Wenzhi Luo.  
	\end{abstract}

	\author[C. Li and Z. Qi]{Changlin Li and Zhi Qi}
	\address{School of Mathematical Sciences\\ Zhejiang University\\ Hangzhou, 310027\\ China}
	\email{12135012@zju.edu.cn, zhi.qi@zju.edu.cn}
	
	\thanks{The second author was supported by National Key R\&D Program of China No. 2022YFA1005300.}
	
	
	\maketitle
	
		\begin{spacing}{1.4}
		{  \tableofcontents}	
	\end{spacing}

		\section{Introduction}\label{sec: intro}
	
	Let $\bfF = \bfQ(\sqrt D)$ be a fixed real quadratic field, with discriminant $D > 1$. Let $\SO $, $\SO^+$, and $\RU$ denote the ring of integers, the set of totally positive integers, and  the group of units, respectively. For simplicity, assume that the narrow class number $h_{\bfF} ^{+} = 1$  and $\bfF \neq \bfQ (\sqrt{2})$ (thus $D$ is  prime and  $D \equiv 1 (\mathrm{mod}\, 4)$; see \cite[\S 26.8]{Hasse-NT}).
	
	  Let $ H_k $ be a Hecke orthonormal basis of $S_k(\mathrm{SL}_2 (\SO))$---the space of Hilbert modular cusp forms  of parallel even weight $(k,k)$ with respect to the Hilbert modular group $\mathrm{SL}_2 (\SO)$.  For $f\in H_k$, let $\omega_f$ be its harmonic weight and let $\lambda_f(\vnu )$ be its Hecke eigenvalue  at $\vnu\in \SO^+/\RU^2$. It is well-known that $ \lambda_f(\vnu) $ are real-valued. 
	
	In 1977, Asai \cite{Asai-1977} introduced the $L$-function
	\begin{equation*}
		L(s,\As(f))=\zeta(2s)\sum_{n\in\BZ_+}\lambda_f(n)n^{-s},\qquad \operatorname{Re}(s)>1,
	\end{equation*}
and he proved the  analytic continuation, the functional equation, and the Euler product for $L(s,\As(f))$. 

In 2024,  Luo \cite{Luo-Asai} initiated the study of the analytic theory of the family of central $L$-values $L(1/2,\As(f))$ and, by the large-sieve approach, he established the sharp mean Lindel\"of upper bound: 
	\begin{equation}\label{1eq: 2nd moment bound, Luo}
	\sumx_{f\in H_k}  | L  (1/2, \As(f)  )  |^2 \Lt_{\bfF, \vepsilon}  k^{2+\vepsilon} ,
\end{equation}
where  the superscript $\star$ restricts the sum to cuspidal Asai lifts $\As(f)$.

Recently, the authors \cite{LQi-Asai-LS} improved Luo's large sieve inequality for $\As (f)$ and  proved a non-trivial bound for convoluted Asai $L$-functions: 
\begin{equation}\label{1eq: 2nd moment bound}
	\sumx_{f\in H_k}  | L  (1/2, \As(f)\times  \phi  )  |^2 \Lt_{\bfF, \phi, \vepsilon}  k^{7/2+\vepsilon} ,
\end{equation}
where  $\phi $ is a fixed Hecke--Maass cusp form for $\SL_2 (\BZ)$. 

In this paper, we return to the study of the second moment of $L  (1/2, \As(f)  )$, for which we establish the following asymptotic formula. % Our main result is the following asymptotic formula. 

\begin{theorem}\label{thm: 2nd moment}
	We have 
	\begin{equation}\label{1eq: main asymp}
		\sum_{f\in H_k}\omega_f L(1/2,\As(f))^2 =  \big(  \sqrt{D}  P_3 ( \log  {k \sqrt D}   ) + C_D \big)  k^2 D + O_{\bfF,\vepsilon}\big(k^{3/2 + \vepsilon}\big),
	\end{equation}
where $P_3 (X)$ is an $\bfF$-independent polynomial of degree $3$ and leading coefficient $1 / 3$ and $C_D$ is a $D$-dependent constant, explicitly given in {\rm\eqref{3eq: Q3(X)}}, {\rm\eqref{3eq: P = Q}}, and {\rm\eqref{3eq: CD}}. 
\end{theorem}

\begin{remark}
	As $L  (1/2, \As(f)     )$ is real-valued, for brevity,  $ |L(1/2,\As(f))|^2 $ in {\rm\eqref{1eq: 2nd moment bound, Luo}} is replaced by $ L(1/2,\As(f))^2 $ in {\rm\eqref{1eq: main asymp}}. 
\end{remark}

\begin{remark}
	Luo restricted the sum in {\rm\eqref{1eq: 2nd moment bound, Luo}} to those cuspidal $\mathrm{As} (f)$   to avoid the possible simple pole of $ L (s, \mathrm{As} (f)) $. This seems inessential as  we may suitably choose the test function  so that  the residual term in the approximate functional equation  is annihilated. See \cite[Theorem 5.3]{IK} and \S {\rm\ref{sec: Asai}}. 
\end{remark}

\begin{remark}
	Note that if $\bfF / \BQ$ were replaced by the split quadratic algebra $\BQ \times \BQ$, then the analogue of   $L (1/2,  \mathrm{As}(f) )$ would be the Rankin--Selberg   $L (s, h \times h) = \zeta (s) L (s, \mathrm{Sym}^2 h)$, for $ h \in S_{k}  (\mathrm{SL}_2 (\BZ))$.  However, there is no result for $ L (s, \mathrm{Sym}^2 h) $ in the literature like Theorem {\rm\ref{thm: 2nd moment}}; for the state-of-the-art result on $\mathrm{Sym}^2 h$, we refer the reader to {\rm\cite{Khan-Sym2-Non-vanishing,Khan-Young-Sym2}}. 
%	There are 2 major differences between the second moments of $  L (1/2, \As(f))$ and $ L (1/2, \mathrm{Sym}^2 h) ${\rm:}
Moreover, let us emphasize two technical points: 
	\begin{itemize}
		\item [{\rm(1)}] in the case of $\As(f)$, the off-diagonal terms are mainly of  analytic nature, since the moduli $c$ of the Kloosterman sums are $O (k^{\vepsilon})$ (see \eqref{eq: range of c});
		\item [{\rm(2)}] in the case of $\mathrm{Sym}^2 h$, the  constant corresponding to $C_D$ vanishes due to cancellation {\rm(}see Remark {\rm\ref{rem: constant CD}}{\rm)}. 
	\end{itemize}
\end{remark}

\subsection*{Non-vanishing of Asai Central $L$-values}

Further, it is relatively easier to prove the following asymptotic formulae. 

\begin{theorem}\label{thm: 1st moment}
	We have 
	\begin{equation}\label{1eq: 0th moment}
		\sum_{f\in H_k}\omega_f  =  { k^2} D \sqrt{  D  } +  
		O_{ \bfF}(k  ), 
	\end{equation}
	\begin{equation}\label{1eq: asymp, 1st}
		\sum_{f\in H_k}\omega_f L(1/2,\As(f)) =   
		\bigg( 
		 \log\bigg(\frac{k \sqrt D}{16\pi^2  }\bigg)  + \gamma
		\bigg) {  k^2} D \sqrt{  D  }
		+
		O_{ \bfF}(k \log k),
	\end{equation}
	where $\gamma$ is the Euler constant.  
\end{theorem}

\begin{remark}
	In view of {\rm\eqref{1eq: main asymp}} and {\rm\eqref{1eq: asymp, 1st}}, according to  \cite{Conrey-FKRS-Moments},  the symmetry type of $\As (f)$ is symplectic---the same as $\mathrm{Sym}^2 h$. 
\end{remark}

\begin{coro} 
	We have 
	\begin{equation}\label{1eq: non-vanishing}
		\sum_{\substack{f\in H_k\\ L(1/2,\As(f))\neq0}}\omega_f \geqslant  \frac{3 - \vepsilon}{\log k} \sum_{f\in H_k}\omega_f,
	\end{equation} 
for any $\vepsilon > 0$. 
\end{coro}

	\begin{proof}
By Theorems \ref{thm: 2nd moment} and  \ref{thm: 1st moment}, we have  
\begin{equation*}%\label{7eq: second moment main for Cauchy}
	\sum_{f\in H_k}\omega_fL(1/2,\As(f))^2
	=
	\bigg(\frac{D^{3/2}}{3}+o(1)\bigg)k^2\log^3 k , 
\end{equation*} 
	\begin{equation*}%\label{7eq: first moment main for Cauchy}
		\sum_{f\in H_k}\omega_fL(1/2,\As(f))
		=
		\big(D^{3/2}+o(1)\big)k^2\log k.
	\end{equation*} 
On substituting these into the Cauchy inequality: 
\begin{equation*}%\label{7eq: Cauchy nonvanishing}
	\bigg( 
	\sum_{f\in H_k}\omega_fL(1/2,\As(f))
	\bigg)^2
	\leqslant \bigg(\sum_{\substack{f\in H_k\\ L(1/2,\As(f))\neq0}}\omega_f \bigg)   
	\bigg(\sum_{f\in H_k}\omega_fL(1/2,\As(f))^2\bigg),  
\end{equation*} 
	we obtain  
	\begin{equation*}%\label{1eq: asymp, non-vanishing}
		\sum_{\substack{f\in H_k\\ L(1/2,\As(f))\neq0}}\omega_f \geqslant  \big(3D^{3/2} - o(1)  \big) \frac{k^2}{\log k}.
	\end{equation*}
From this and \eqref{1eq: 0th moment}, we deduce readily \eqref{1eq: non-vanishing}. 
\end{proof}

\subsection*{Notation}   
By $F \Lt G$ or $F = O (G)$ we mean that $|F| \leqslant c G$  for some constant $c  > 0$, and by $F \asymp G$ we mean that $F \Lt G$ and $G \Lt F$. We write $F \Lt_{\valpha, \beta, ...} G $ or $  F = O_{\valpha, \beta, ...} (G) $ if the implied constant $c$ depends on $\valpha$, $\beta$, ....

By `negligibly small' we mean  $O_A (k^{-A})$ (or $  O_A (\vkappa^{-A})$ as we shall set $\vkappa = k-1$) for arbitrarily large but fixed $A > 0$. 

Throughout the paper,  $\upvarepsilon $ is arbitrarily small and its value may differ from one occurrence to another.

		\section{Preliminaries}\label{sec: preliminaries}
	
	\subsection{Basic Notation}\label{subsec: basic notation}
	
Let $\bfF  \hookrightarrow \BR^2$ via its two real embeddings.	
Let $\xi'$ denote the Galois conjugate of $\xi\in\bfF$.  Define the norm and the trace 
	\begin{equation*}
		N(\xi)=\xi\xi',\qquad \Tr(\xi)=\xi+\xi'.
	\end{equation*}
%Let $\SO $, $\SO^+$, and $\RU$ denote the ring of integers, the set of totally positive integers, and  the group of units, respectively.     
For $x \in \mathbf{R}^2$, we write
	\begin{equation*}
	  e_{\bfF}[x]=e(\Tr( x / \sqrt{D}) ) = \exp(2\pi i \Tr(x / \sqrt{D})) . 
	\end{equation*} 
Note that  the different ideal is $( \sqrt{D})$.   For $\vnu_1,\vnu_2\in\SO$ and $c\in\SO\smallsetminus\{0\}$, define the Kloosterman sum
	\begin{equation*}
		S_{\bfF}(\vnu_1,\vnu_2;c)=\sumx_{\valpha\, (\mathrm{mod}\, c)}
		e_{\bfF}\bigg[\frac{\vnu_1\valpha+\vnu_2 \widebar{\valpha} }{c} \bigg],
	\end{equation*}
	where the star $\star$ denotes $(\valpha,c)=1$, and $\widebar{\valpha}$ is the inverse of $\valpha$ modulo $c$. 

\subsection{Hilbert Cusp Forms} 

For $k\in 2\BZ_+$, every $f\in H_k$ has Fourier expansion
\begin{equation*}
	f( {z})=\sum_{\vnu\in\SO^+ } a_f(\vnu)e(\Tr(\vnu  {z} / \sqrt{D}) ), \qquad  {z}  \in \bfH^2,  
\end{equation*}
where $ \bfH $ is the hyperbolic upper half-plane.  We have  $a_f(\epsilon^2\vnu)=a_f(\vnu)$ for any $\epsilon\in\RU$. 
It is well-known that the Fourier coefficients $a_f (\vnu)$ and the Hecke eigenvalues $\lambda_f (\vnu) $ are related by 
\begin{equation*}
	a_f(\vnu)=a_f(1)\lambda_f(\vnu)N(\vnu)^{(k-1)/2}.
\end{equation*}
Define the harmonic weight 
\begin{align*}%\label{2eq: harmonic weight}
	\omega_f = \frac{\Gamma (k)^2 D^{k+1} }{(4\pi)^{2(k-1)} } |a_f(1)|^2. 
\end{align*} 
\delete{Note that $ \omega_f $ only plays a very minor role because we have a direct extension of the bounds of Iwaniec, Hoffstein, and Lockhart \cite{Iwaniec-L(1),HL-L(1)}
\begin{align}
	k^{-\vepsilon} \Lt	\omega_f \Lt k^{\vepsilon} . 
\end{align}}

	\subsection{Asai $L$-functions}\label{sec: Asai}
Recall that the Asai $L$-function $L (s, \mathrm{As}(f))$ is defined by
\begin{equation}\label{2eq: L(s, As)}
	L(s,\As(f) ) = \zeta(2s) \sum_{n \in \BZ_{+}}  \lambda_f(n) n^{-s}, \qquad \text{($\mathrm{Re} (s)>1$)}.
\end{equation}
Asai \cite{Asai-1977} proved that $L(s,\As(f) ) $ has  analytic continuation to the whole $s$-plane with a possible simple pole at $s =  1$, and satisfies the functional equation
\begin{equation}\label{2eq: FE}
	\Lambda (s,\As(f) ) = \Lambda (1-s,\As(f) ),
\end{equation}
where $\Lambda (s, \As(f) ) = D^{s/2} \gamma_{k - 1} (s ) L (s,\As(f) )$, with 
\begin{equation}\label{2eq: gamma k(s)}
	\gamma_{\vkappa}(s ) = (2\pi)^{-2s} \Gamma(s+ \vkappa) \Gamma(s) .
\end{equation} 

According to \cite[Theorem 5.3]{IK}, we deduce from \eqref{2eq: L(s, As)} and \eqref{2eq: FE} the approximate functional equation: 
	\begin{equation}\label{eq: AFE}
	L(1/2,\As(f))=2 \sum_{n \in \BZ_{+}} \frac{\lambda_f(n)}{\sqrt{n} }V_{k-1} \bigg(\frac{n}{\sqrt D}\bigg),
\end{equation}
with  
\begin{equation}\label{eq: V(y)}
	V_{\vkappa} (y)=\frac{1}{2\pi i}\int_{(3)}  \zeta(1+2 v)
	\frac{\gamma_{\vkappa} (1/2+ v )}{\gamma_{\vkappa} (1/2 )}  G (v) y^{-v} \frac{\nd v}{v}, 
\end{equation}
where we choose the test function
\begin{align}\label{2eq: G(s)}
	 G (v) = (1 - 4 v^2 )  \exp (v^2), 
\end{align}
so that the possible simple poles of $ \Lambda (s, \As(f)) $ at $s = 0, 1$ do not contribute a residual term.    It is clear that the weight function $  V_{\vkappa} (y)$ is real-valued and, by \cite[Proposition 5.4]{IK} (it is easy to see that the proof therein works for our choice of test function), its derivatives satisfy the following bound:
\begin{align} \label{2eq: bound for Vk(y)}
	y^{j} V_{\vkappa}^{(j)} (y) \Lt_{j, A} \min \big\{ 1,  (    {\vkappa} / {y}  )^{ A} \big\} , 
\end{align}
for any $j, A \geqslant 0$. Thus we may effectively restrict the summation in \eqref{eq: AFE} to the range $ n \Lt k^{1+\vepsilon}$ at the cost of a negligible error.   

\subsection{Petersson Trace Formula for Hilbert Cusp Forms}\label{subsec: Petersson formula}

For $\vnu_1,\vnu_2\in\SO^+$, the Petersson trace formula for Hilbert cusp forms for  $\SL_2(\SO)$ reads (see  \cite{Luo-2003-Hilbert,Luo-Asai}): 
\begin{equation}\label{eq: Petersson}
	\begin{split}
		\frac{1}{(k-1)^2 D}\sum_{f\in H_k}\omega_f\lambda_f(\vnu_1)\lambda_f(\vnu_2)
		& =  \sqrt{D} \delta(\vnu_1\sim\vnu_2) \\
		&+2\pi^2 \sum_{c\, \in\SO\smallsetminus\{0\}}
		\frac{S_{\bfF}(\vnu_1,\vnu_2;c)}{|N(c)|}
		NJ_{k-1}\bigg(\frac{4\pi\sqrt{\vnu_1 \vnu_2}}{|c|}\bigg),
	\end{split}
\end{equation}
where $\delta (\vnu_1 \sim \vnu_2)$ is the Kronecker $\delta$ symbol for $ \vnu_1 \sim \vnu_2 $ (by definition, this stands for $(\vnu_1) = (\vnu_2)$ or $\vnu_1 / \vnu_2 \in \mathrm{U}^2 $), and
\begin{equation*}
	NJ_{k-1}\bigg(\frac{4\pi\sqrt{\vnu_1 \vnu_2}}{|c|}\bigg)
	=J_{k-1}\bigg(\frac{4\pi\sqrt{\vnu_1 \vnu_2}}{|c|}\bigg)
	J_{k-1}\bigg(\frac{4\pi\sqrt{\vnu_1'\vnu_2'}}{|c'|}\bigg).
\end{equation*}

\begin{remark}\label{rem: U/pm1}
	Note that the sum over $ \epsilon \in \mathrm{U} $ in \cite[(9)]{Luo-2003-Hilbert} or \cite[(2)]{Luo-Asai} should instead be taken over  $ \epsilon \in \mathrm{U} /\{ \pm 1 \} $ since Luo inadvertently miscounted the center $\{ \pm \boldsymbol{1}_2 \}$ of $ \SL_2 (\SO) $. Actually, the formula \cite[(9)]{Luo-2003-Hilbert} for $\mathbf{F} = \mathbf{Q}$ differs slightly from the classical Petersson formula as in \cite[(14.15)]{IK}. 
\end{remark}

\begin{remark}
	For simplicity,  we have combined the sums over $\epsilon \in \mathrm{U} / \{ \pm 1 \}$ {\rm(}see Remark {\rm\ref{rem: U/pm1}}{\rm)} and $ c \in  (\SO   \smallsetminus \{0\}) / \mathrm{U}  $ in Luo's Petersson formula \cite[(2)]{Luo-Asai} {\rm(}Luo's notation $\SO^{\times}$ denotes $\SO \smallsetminus\{0\} $ but it usually means $\mathrm{U}$ in the literature{\rm)} into the sum over $ c \in \SO   \smallsetminus \{0\} $ in \eqref{eq: Petersson}. 
\end{remark}

\subsection*{Abbreviation} For brevity, we shall henceforth  write  
\begin{align}
	\vkappa = k -1 . 
\end{align}

\subsection{Properties of Bessel Functions} Next, we recollect some results of the Bessel function $J_{\vkappa} (x)$ of  positive argument $x \in \BR_+$ and integral order $\vkappa \in \mathbf{Z}_{+}$. 

Firstly, in some preliminary analysis, we shall need the uniform crude bounds 
\begin{align}\label{2eq: prelim bounds}
	 J_{\vkappa}(x)  \Lt \min \bigg\{ 1,  \Big( \frac{e x}{2 \vkappa}\Big)^{\vkappa} \bigg\}, 
\end{align}
which, by trivial estimation (along with the Stirling  formula), may be deduced from the Bessel and Poisson  integral representations (see \cite[2.2 (2), 3.3 (1)]{Watson}):  \begin{align*}%\label{2eq: Bessel's integral}
	J_{\vkappa}(x) = \frac{1}{\pi} \int_{0}^{\pi} \cos(\vkappa \theta - x \sin\theta) \nd \theta.
\end{align*}
\begin{align*}
	J_{\vkappa}(x) = \frac{  (x/2)^{\vkappa} }{ \sqrt{\pi}\, \Gamma(\vkappa+1/2)  } \int_{0}^{\pi } \cos(x \cos\theta) \sin^{2\vkappa}\theta \, \nd \theta.
\end{align*}
It is clear from \eqref{2eq: prelim bounds}  that $J_{\vkappa}(x)$ is negligibly small unless $x > \vkappa/2$.

Our analysis will rely crucially on the  Macdonald integral representation (\cite[13.7(1)]{Watson}):
\begin{equation}\label{2eq: Macdonald}
	J_{\vkappa} (x) J_{\vkappa} (y) = \frac{1}{2\pi i^{\vkappa+1}} \int_{ {-\infty}}^{\infty}  { \exp \bigg( \frac {i} {2} \bigg( {r}  + \frac{x^2 + y^2}{  r}\bigg) \bigg) J_{\vkappa} \Big( \frac{xy}{ r}   \Big) } \frac{\nd r}{r},
\end{equation}
for $x, y \in \BR_{+}  $. 

Our analysis for zero frequencies will require the following   Mellin--Barnes integral representation \cite[(10.9.29)]{NIST}:
\begin{equation}\label{eq: MB Bessel square}
	J_{\vkappa}(x)^2=\frac{1}{2\pi i}\int_{(\sigma)}
	\frac{\Gamma(1+2s)\Gamma(\vkappa-s)}{\Gamma(1+s)^2\Gamma(\vkappa+1+s)}
	\Big(\frac{x}{2} \Big)^{2s}\nd s,
	\qquad - \frac 1 2 <\sigma<\vkappa. 
\end{equation}

Moreover, as a special case of \cite[(10.22.57)]{NIST},  
\begin{equation}
	\label{2eq: integral of J(x)2}
	\int_0^{\infty} J_{\vkappa}(x)^2 \frac {\nd x} {x} = \frac 1 {2\vkappa}. 
\end{equation}

Finally, we record here a variant of \cite[Lemma B.1]{Qi-GL(3)}, which collects several results for $ J_{\vkappa} (\vkappa x) $ derived from the Olver asymptotic formula  \cite{Olver-1,Olver-Bessel}. 

		\begin{lem}\label{lem: Olver}
		For $x > 1$ define
		\begin{align}\label{2eq: gamma}
			\gamma (x) = \sqrt{x^2-1} - \mathrm{arcsec}\, x  . 
		\end{align}
		
		{\rm (1)} For $|x-1| \leqslant 1/\vkappa^{2/3}$,  we have $ J_{\vkappa} (\vkappa x) = O (1 /\vkappa^{1/3}) $. 
		
		{\rm (2)} For $1/2 < x < 1 - 1/\vkappa^{2/3} $, we have
		\begin{align}\label{12eq: Bessel, smaller, K}
			J_{\vkappa} (\vkappa x) = O  \bigg(  \frac{\exp  \big(  \!  -    \vkappa/3 \cdot (2-2x) ^{3/2}   \big) }{  \vkappa^{1/2} (1-x)^{1/4}  } \bigg).
		\end{align} 
		
		{\rm(3$'$)} For $ 1 + 1/ \vkappa^{2/3} < x \Lt \vkappa^{13/3}$, we have
		\begin{align}\label{Beq: Olver}
			J_{\vkappa} (\vkappa x) = \sqrt{2} \sum_{ \pm } \frac { \exp ( \pm i \vkappa \gamma (x) ) } {\vkappa^{1/2} (x^2-1)^{1/4} } W_{\pm} (\vkappa \gamma (x))  + O\bigg(\frac{1}{\vkappa x } \bigg) ,
		\end{align}
		in which $\gamma^j W_{\pm}^{(j)}(\gamma)\Lt_j 1$ for $\gamma \Gt 1$. 
		 
	\end{lem}

\begin{remark}
	By (3) in \cite[Lemma B.1]{Qi-GL(3)},  for $ 1 + 1/ \vkappa^{2/3} < x \leqslant 2$, the error term in {\rm\eqref{Beq: Olver}} may be improved into $ O \big(1 / \vkappa^{7/6} (x-1)^{1/4} \big) $, but it will be sufficient to use the uniform error bound $ O (1/\vkappa x) $.  
\end{remark}
	
	\section{Initial Reductions}\label{sec: initial steps}
	
	Consider 
	\begin{equation}\label{3eq: definition of Mk}
		\EuScript{M}_k = \sum_{f\in H_k}\omega_f L(1/2,\As(f))^2.
	\end{equation}
By the approximate functional equation for $L(1/2,\As(f))$ as in \eqref{eq: AFE},  
\begin{equation*}
	\EuScript{M}_k= 4 \mathop{\sum\sum}_{n_1,n_2\in\BZ_+}
	\frac{1}{\sqrt{n_1n_2}}V_{k-1} \bigg(\frac{n_1}{\sqrt D}\bigg)
	V_{k-1} \bigg(\frac{n_2}{\sqrt D}\bigg)
	\sum_{f\in H_k}\omega_f\lambda_f(n_1)\lambda_f(n_2).
\end{equation*}

\subsection{Application of the Petersson Formula} 
By the Petersson trace formula in \eqref{eq: Petersson}, with $\vnu_1=n_1$ and $\vnu_2=n_2$, we obtain
\begin{align}\label{3eq: M=D+O}
	\EuScript{M}_k = 4    \vkappa^2 D \bigg( \sqrt{D} \cdot \EuScript{D}_{\vkappa} +  \pi i^{\vkappa+1}   \sum_{c\, \in\SO\smallsetminus\{0\}}\frac{\EuScript{O}_{\vkappa}(c)}{|N(c)|} \bigg), \qquad \vkappa = k-1, 
\end{align}
where 
\begin{align}\label{3eq: Dk}
	 \EuScript{D}_{\vkappa} = \sum_{n\in\BZ_+}\frac{1}{n} V_{\vkappa} \bigg(\frac{n}{\sqrt D}\bigg)^2, 
\end{align}
\begin{equation}\label{eq: Ok(c)}
	\EuScript{O}_{\vkappa}(c)= 2\pi i^{\vkappa+1} \! \mathop{\sum\sum}_{n_1, n_2\in\BZ_+} \! \frac{V_{\vkappa} (n_1, n_2)}{\sqrt{n_1 n_2}} S_{\bfF}(n_1, n_2;c) NJ_{\vkappa}\bigg(\frac{4\pi\sqrt{n_1 n_2}}{|c|}\bigg), 
\end{equation}
and  
\begin{align}\label{3eq: Vk(x1,x2)}
	V_{\vkappa} (x_1, x_2) =	V_{\vkappa} \bigg(\frac{x_1}{\sqrt D}\bigg)V_{\vkappa}\bigg(\frac{x_2}{\sqrt D}\bigg) . 
\end{align}
In view of the bounds for $ V_{\vkappa} (y) $ and $J_{\vkappa} (x)$ as in \eqref{2eq: bound for Vk(y)} and \eqref{2eq: prelim bounds}, we may restrict the $c$-sum %and the double $n_1, n_2$-sum 
effectively to 
\begin{equation}\label{eq: range of c}
	|c|,  |c'| \Lt k^{\vepsilon};  %\quad n_1, n_2 \Lt k^{1+\vepsilon}; 
\end{equation}
outside this range, the total contribution is negligible.  

\subsection{Application of the Poisson Summation} 
Now we apply the Poisson summation formula to both the $n_1$- and $n_2$-sums in  \eqref{eq: Ok(c)} modulo $N(c) \cdot \BZ$,   obtaining 
\begin{align}\label{3eq: O(c)=Z(c)+O(c)}
	\EuScript{O}_{\vkappa}(c)= \EuScript{Z}_{\vkappa}(c) + \widetilde{\EuScript{O}}_{\vkappa}(c) , 
\end{align}
where $ \EuScript{Z}_{\vkappa}(c)  $ is the zero frequency:
\begin{align}\label{3eq: Z(c)}
	\EuScript{Z}_{\vkappa}(c) = T (0, 0; c) I_{\vkappa} (0, 0; c), 
\end{align}
and  $ \widetilde{\EuScript{O}}_{\vkappa}(c)  $ is   the dual sum on $\BZ^2 \smallsetminus \{(0, 0)\} $:
\begin{align}\label{3eq: O(c), dual}
	\widetilde{\EuScript{O}}_{\vkappa}(c) = \mathop{\mathop{\sum \sum}_{ n_1   , n_2     \in \BZ  }}_{(n_1   , n_2   ) \neq (0, 0)} T (n_1   , n_2   ; c) I_{\vkappa} (n_1   , n_2   ; c) , 
\end{align}
with 
	\begin{equation}\label{3eq: T(m1,m2;c)}
	T(n_1   ,n_2   ;c)=\frac{1}{N(c)^2}\mathop{\sum\sum}_{a_1,a_2  (\mathrm{mod}\, N(c) \BZ)}
	S_{\bfF}(a_1,a_2;c)e\bigg(\frac{a_1 n_1   +a_2 n_2   }{|N(c)|}\bigg),
\end{equation}
\begin{equation}\label{3eq: Ik(m1,m2;c)}
	I_{\vkappa} (n_1   , n_2   ; c) = 2\pi i^{\vkappa+1} \! \int\!\!\int_{\BR_+^2} \! V_{\vkappa} (x_1, x_2)  NJ_{\vkappa}\bigg(\frac{4\pi\sqrt{x_1 x_2}}{|c|}\bigg) e \bigg( \! - \frac {n_1    x_1 + n_2    x_2} {|N(c)|}  \bigg) \frac {\nd x_1 \nd x_2} {\sqrt{x_1 x_2}} .
\end{equation}
Further, in order to facilitate our analysis for the integral $ I_{\vkappa} (n_1   , n_2   ; c) $, we introduce dyadic partitions for both the $x_1$- and $x_2$-integrals; the partitioned integral reads
\begin{align}\label{3eq: IX(n;c)}
	\begin{aligned}
		  I_{X_1, X_2} & (n_1   , n_2   ; c) = \frac{2\pi i^{\vkappa+1}}{\sqrt{X_1 X_2}} \\
	& \cdot \int\!\!\int_{\BR_+^2} \! \varww_{X_1, X_2} (x_1, x_2)  NJ_{\vkappa}\bigg(\frac{4\pi\sqrt{x_1 x_2}}{|c|}\bigg) e \bigg( \! - \frac {n_1    x_1 + n_2    x_2} {|N(c)|}  \bigg)   {\nd x_1 \nd x_2},
	\end{aligned}
\end{align} 
for $ \varww_{X_1, X_2} (x_1, x_2) \in C_c^{\infty} ([X_1, 2 X_1] \times [X_2, 2 X_2]) $ such that $ X_1^{j_1} X_2^{j_2} \varww_{X_1, X_2}^{(j_1, j_2)} (x_1, x_2) \Lt_{j_1, j_2}  1$. More explicitly, 
\begin{align*}
	  \varww_{X_1, X_2} (x_1, x_2) = \frac {\sqrt{X_1 X_2}} {\sqrt{x_1 x_2}} V_{\vkappa} (x_1, x_2)  \varvv \Big( \frac {x_1} {X_1} \Big) \varvv \Big( \frac {x_2} {X_2} \Big) , 
\end{align*} 
for a suitable (real-valued) fixed $\varvv (x) \in C_c^{\infty} [1, 2]$.   Given that $|c|,  |c'| \Lt k^{\vepsilon}$ as in \eqref{eq: range of c},  it follows again from the bounds for $ V_{\vkappa} (y) $ and $J_{\vkappa} (x)$ in \eqref{2eq: bound for Vk(y)} and \eqref{2eq: prelim bounds} that we may restrict the parameters $X_1$ and $X_2$ effectively to the range 
\begin{equation}\label{3eq: range of X}
k^{1-\vepsilon} \Lt	X_1, X_2 \Lt k^{1+\vepsilon} . 
\end{equation}

	\subsection{Application of the Macdonald Integral Representation}\label{subsec: Macdonald representation}

Let $|c|,  |c'| \Lt k^{\vepsilon}$ as in \eqref{eq: range of c}.  For the dual sum $ \widetilde{\EuScript{O}}_{\vkappa}(c) $ defined in \eqref{3eq: O(c), dual}, let the contribution from $I_{X_1, X_2}   (n_1   , n_2   ; c)$ be denoted by $ \widetilde{\EuScript{O}}_{X_1, X_2} (c) $, that is, 
\begin{align}\label{3eq: OX(c), dual}
	\widetilde{\EuScript{O}}_{X_1, X_2}   (c) = \mathop{\mathop{\sum \sum}_{ n_1   , n_2     \in \BZ  }}_{(n_1   , n_2   ) \neq (0, 0)} T (n_1   , n_2   ; c) I_{X_1, X_2}  (n_1   , n_2   ; c) . 
\end{align}
Of course, the trivial bound  \begin{align}\label{3eq: bound for T (n;c)}
	 T (n_1, n_2; c) = O (|N(c)|) = O (k^{\vepsilon})  
\end{align}  will be sufficient for our purpose.  Next we apply the Macdonald integral formula in \eqref{2eq: Macdonald} in the form
 	\begin{equation}\label{eq: Macdonald Asai form}
 	NJ_{\vkappa}\bigg(\frac{4\pi\sqrt{x_1 x_2}}{|c|}\bigg)
 	=\frac{1}{2\pi i^{\vkappa+1}}\int_{-\infty}^{\infty}
 	e\bigg(\frac{x_1 x_2 }{|N(c)|r}+\Tr  |c'/c  |  r \bigg)J_{\vkappa}(4\pi r)\frac{\nd r}{r}.
 \end{equation}
On inserting this into \eqref{3eq: IX(n;c)} and \eqref{3eq: OX(c), dual}, it follows that 	
\begin{equation}\label{eq: integral of L_k}
	\widetilde{\EuScript{O}}_{X_1, X_2}  (c)=\int_{-\infty}^{\infty}
	e (\Tr|c'/c|  r )J_{\vkappa}(4\pi r)
	\EuScript{Q}_{X_1, X_2}(r;c)\frac{\nd r}{r},
\end{equation}
where
\begin{equation}\label{3eq: Q (r; c)} 
	\EuScript{Q}_{X_1, X_2}(r;c)= \mathop{\mathop{\sum \sum}_{ n_1   , n_2     \in \BZ  }}_{(n_1   , n_2   ) \neq (0, 0)} T (n_1, n_2; c) 
	I_{X_1, X_2} (n_1, n_2; r;  |N(c)| ) , 
\end{equation} 
with 
\begin{equation}\label{3eq: IX (n;r;c)}
	I_{X_1, X_2} (n_1, n_2; r;  q) =\frac{1}{\sqrt{X_1 X_2}}\iint
	\varww_{X_1, X_2} (x_1,x_2)e\bigg(\frac{x_1x_2/r-n_1 x_1-n_2 x_2}{q} \bigg)\nd x_1\nd x_2.
\end{equation}
Trivially, it follows from \eqref{3eq: bound for T (n;c)}, \eqref{eq: integral of L_k}, and \eqref{3eq: Q (r; c)} that
\begin{align}\label{3eq: dual O(c)}
	\widetilde{\EuScript{O}}_{X_1, X_2}  (c) \Lt k^{\vepsilon} \! \int_0^{\infty} | J_{\vkappa}(4\pi r) | \cdot   \EuScript{Q}^{\flat}_{X_1, X_2}(r;c)   \frac{\nd r}{r}, 
\end{align}
where 
\begin{align}\label{3eq: dual Q(r;c)}
	\EuScript{Q}^{\flat}_{X_1, X_2}(r;c) = \mathop{\mathop{\sum \sum}_{ n_1   , n_2     \in \BZ  }}_{(n_1   , n_2   ) \neq (0, 0)}  
\big| I_{X_1, X_2} (n_1, n_2; r;  |N(c)| ) \big|. 
\end{align}
By \eqref{2eq: prelim bounds}, we may truncate the $r$-integral at $r = \vkappa / 8\pi$ due to the decay of $ J_{\vkappa}(4\pi r) $ in \eqref{3eq: dual O(c)}.

\subsection{Reductions} 
	
	The rest of this paper will be devoted mainly to the proof of the following asymptotic formulae or bounds for the diagonal sum $ \EuScript{D}_{\vkappa}  $, the  zero frequencies $\EuScript{Z}_{\vkappa}(c) $ (for $c$ rational or not), and the dual sums $ \widetilde{\EuScript{O}}_{X_1, X_2}  (c) $. 
	
	\begin{defn}
	Let $\psi (s) = \Gamma' (s) / \Gamma (s)$ be the di-gamma function. Define the constants
		 \begin{equation}\label{3eq: constants psi}
		 	\psi_{j} = \psi^{(j)} (1/2) =   (-1)^{j+1} j! (2^{j+1}- 1) \zeta (j+1),  
		 \end{equation}
	 for $ j \in \BZ_+${\rm(}see \cite[\S 1.2]{MO-Formulas}{\rm)}, and the Stieltjes constants $\gamma_{j}$  by \begin{equation}\label{eq: Stieltjes constants definition}
	 	\zeta(1+s)=\frac1s+\sum_{j=0}^{\infty}\frac{(-1)^j\gamma_j}{j!}s^j,
	 	\qquad \gamma_0=\gamma. 
	 \end{equation} 
	\end{defn}
	
	\begin{lem}\label{lem: diag}
		We have 
		\begin{equation}\label{3eq: Dk = Q3+CD}
			 \EuScript{D}_{\vkappa} = Q_3 \bigg( \! \log \frac{k \sqrt D}{16\pi^2} \bigg) + C + O_{D, \vepsilon } \bigg(\frac {k^{\vepsilon} } { \sqrt{k}} \bigg), 
		\end{equation}
	for %$ Q_3 (Y)$ is a polynomial of degree $3$, explicitly given by 
	\begin{equation}\label{3eq: Q3(X)}
		Q_3 (X) =  	\frac 1 {12} X^3 + \frac {\gamma} 2 X^2 + \frac {3\gamma^2 - 2 \gamma_1} 4 X - \frac {8\gamma^3 + 60 \gamma \gamma_1 + 18 \gamma_2 + \psi_2 } {24}  ,
	\end{equation}
\begin{equation}\label{3eq: CD, diag}
		C = - \frac 1 {2\pi i} \int_{(1)}  \xi ( 2v ) \xi (1+2v)  G (v)^2 \frac {\nd v} {v^2}, 
\end{equation}
where as usual  $ \xi (s) = \pi^{-s/2} \Gamma (s/2) \zeta (s) $. 
	\end{lem}

\begin{lem}\label{lem: zero-freq, asymp}
	Define 
	\begin{equation}\label{3eq: Zk}
		\EuScript{Z}_{\vkappa} =  \pi i^{\vkappa+ 1} \sum_{c\, \in\BZ\smallsetminus\{0\}}\frac{\EuScript{Z}_{\vkappa}(c)}{c^2 } .
	\end{equation}
We have 
\begin{equation}
	 \EuScript{Z}_{\vkappa} = A_D + O_{ D} \bigg(\frac {\log^3 k } { {k}} \bigg), 
\end{equation}
for 
\begin{equation}\label{3eq: AD}
	A_D =   \frac { 1 } {2 \pi i } \int_{(1)} \xi (2v) \xi (1+2v) G(v )^2  D^{v }  \frac {\nd v} {v^2}.  
\end{equation}
\end{lem}

\begin{lem}\label{lem: zero-freq, bound}
	 We have  
	 \begin{equation}
	 	 \EuScript{Z}_{\vkappa}(c) = O_{\bfF, \vepsilon} \bigg(\frac {k^{\vepsilon} } {\sqrt{k}} \bigg) , 
	 \end{equation}
 for any $ c \in \SO \smallsetminus \BZ $, with $ |c|, |c'| \Lt k^{\vepsilon}$.  
\end{lem}

\begin{lem}\label{lem: dual}
Let  $ k^{1-\vepsilon} \Lt	X_1, X_2 \Lt k^{1+\vepsilon} $. 	We have 
	\begin{equation}
		\widetilde{\EuScript{O}}_{X_1, X_2}  (c)  = O_{\bfF, \vepsilon} \bigg(\frac {k^{\vepsilon} } {\sqrt{k}} \bigg) , 
	\end{equation}
	 for any $ c \in \SO \smallsetminus \{0 \} $, with $ |c|, |c'| \Lt k^{\vepsilon}$. 
\end{lem}

In view of \eqref{3eq: M=D+O} and \eqref{3eq: O(c)=Z(c)+O(c)}, the asymptotic formula in \eqref{1eq: main asymp} in Theorem \ref{thm: 2nd moment} is a direct consequence of the lemmas above, with 
\begin{equation}\label{3eq: P = Q}
	P_3 (X) = 4 Q_3 (X - 2 \log (4\pi)), 
\end{equation}
\begin{equation}\label{3eq: CD}
	C_D = 4 \sqrt{D} C + 4  A_D =   \frac {2} {\pi i }   \int_{(1)}  \xi ( 2v ) \xi (1+2v)  G (v)^2 (D^v - \sqrt{D} ) \frac {\nd v} {v^2}. 
\end{equation}

\begin{remark}\label{rem: constant CD}
	  For the second moment of $ L (1/2, \mathrm{Sym}^2 h) $, for $ h $ in a Hecke orthonormal basis of $   S_{k}  (\mathrm{SL}_2 (\BZ))$, averaged over $ k \asymp K $, Khan obtained integrals in (3.21) and (3.28) in \cite{Khan-Sym2-Non-vanishing} similar to those in {\rm\eqref{3eq: CD, diag}} and {\rm\eqref{3eq: AD}}, but in that case these integrals cancel each other.  However, the cancellation does not occur in our  setting as $D > 1$!   
\end{remark}

\section{Asymptotic for the Diagonal Sums: Proof of Lemma \ref{lem: diag}}

By \eqref{eq: V(y)} and \eqref{3eq: Dk}, we have the double Mellin integral representation
\begin{equation}\label{eq: diagonal exact double Mellin}
	\begin{split}
		\EuScript{D}_{\vkappa}
		= \frac{1}{  (2\pi i)^2} \!   \int \!\!  \int_{(3)} \! %\int_{(3)} \!
		\zeta(1+v_1+v_2)  \zeta (v_1, v_2)
	\frac{\gamma_{\vkappa} ( v_1 , v_2 )}{\gamma_{\vkappa} (1/2 )^2} 
		G(v_1)G(v_2) D^{\textstyle  \frac{v_1+v_2 }{2} } \frac{\nd v_1\nd v_2}{v_1 v_2} ,
	\end{split}
\end{equation}
where 
\begin{align}\label{4eq: gamma k (u1, u2)}
\zeta (v_1, v_2) =	\zeta(1+2v_1)\zeta(1+2v_2)  , \qquad 
	 \gamma_{\vkappa} ( v_1 , v_2 ) = \gamma_{\vkappa} ( 1/2 + v_1 ) \gamma_{\vkappa} ( 1/2 + v_2 ). 
\end{align}
%where $\zeta(1+v_1+v_2)$ is from the $n$-sum in \eqref{3eq: Dk}. 
Let us shift the integral contours to $\mathrm{Re}(v_1) = \mathrm{Re}(v_2) = \vepsilon$. Note that, at the cost of a negligible error,  we may restrict the double integral to $ |\mathrm{Im} (v_1) |, \, |\mathrm{Im} (v_2) | \leqslant \log k $ due to the vertical exponential decay of $ G(v) $ (see \eqref{2eq: G(s)}).   It follows from the Stirling formula \eqref{eqapp: Stirling, quot} that for $\mathrm{Re} (v) = \vepsilon$ and  $|\mathrm{Im} (v)|  \leqslant \log k$ we have 
\begin{equation*}
	\frac{\Gamma(\vkappa+1/2+v)}{\Gamma(\vkappa+1/2)}
	= k^v \bigg(1+O_{\vepsilon} \bigg(\frac{\log^2 k}{k}\bigg)\bigg). 
\end{equation*}
 Thus, in view of \eqref{2eq: gamma k(s)}, \eqref{2eq: G(s)},  \eqref{eq: diagonal exact double Mellin}, and \eqref{4eq: gamma k (u1, u2)}, 
\begin{equation}
\begin{split}
		\EuScript{D}_{\vkappa} \!
	= \! \frac{1}{  (2\pi i)^2} \!    \int \!\!   \int_{(\vepsilon)} \! \! %\int_{(3)} \!
	\zeta(1+v_1+v_2)  \xi (v_1, v_2) 
	G(v_1)G(v_2) K^{ {v_1+v_2 } }   \frac{\nd v_1\nd v_2}{v_1 v_2}  
	 + O_{\vepsilon}  \bigg(  \frac {k^{\vepsilon}} {k}  \bigg) & ,
\end{split}
\end{equation}
where 
\begin{align}
	K = \frac {k \sqrt{D}  }  {4\pi} , 
\end{align}  
\begin{align}
	\xi (v_1, v_2) = \xi (1+2v_1) \xi (1+2v_2) . 
\end{align}
Now we shift the integral contours to  $\mathrm{Re}(v_1) = \mathrm{Re}(v_2) = -1/2 + \vepsilon$. Clearly,  the resulting double integral is again $ O_{\vepsilon} (k^{\vepsilon} / k) $, and it is left to calculate the residual contribution.
To this end, let us first move the $v_1$-contour and cross the poles at $v_1 = 0$ and $v_1 = -v_2$.  

Since 
\begin{equation*}
	\frac {\xi (1+2v)} {v} = \frac 1 {2 v^2} +     \frac {{\gamma}   - \log (4 {\pi}) } {2 v} + O (1),  
\end{equation*}
as $v \ra 0$ (due to $\psi (1/2) = - \gamma - 2\log 2$ (see \cite[\S 1.2]{MO-Formulas})),  the double pole at $v_1 = 0$ yields the residual contribution:
\begin{align}\label{4eq: residual integral}
	\frac 1 {4 \pi i}  \int_{(\vepsilon)} \big( (\log K + \gamma - \log (4\pi) ) \zeta (1+v) + \zeta' (1+v)  \big) \xi (1+2 v) G(v) K^{v}  \frac {\nd v} {v} . 
\end{align}
Next we move the contour of this integral to $ \mathrm{Re} (v) = -1/2 +\vepsilon $, obtaining the error term $ O_{\vepsilon} (k^{\vepsilon} /\sqrt{k}) $ and the main residual term from the pole  of {fourth} order at $ v = 0$.  It is evident from \eqref{4eq: residual integral} that the residue is indeed a cubic polynomial in $\log  {K}$, but it requires some careful  (albeit routine) calculations to prove the explicit expression as in \eqref{3eq: Dk = Q3+CD} and \eqref{3eq: Q3(X)}.

Finally, the contribution from the simple pole at $v_1 = - v_2  $ is equal to 
\begin{equation}
	C = - \frac 1 {2\pi i} \int_{(\vepsilon)}  \xi ( 2v ) \xi (1+2v)  G (v)^2 \frac {\nd v} {v^2} ,
\end{equation}
and we arrive at \eqref{3eq: CD, diag} by shifting the integral contour to $\mathrm{Re} (v) = 1$. Note that $G (1/2) = 0$ by its definition in \eqref{2eq: G(s)}.

\section{Analysis for the Zero Frequencies: Proof of Lemmas \ref{lem: zero-freq, asymp} and \ref{lem: zero-freq, bound}} 

In this section, we consider the zero frequency  \begin{align*}
	 \EuScript{Z}_{\vkappa}(c) = T (0, 0; c) I_{\vkappa} (0, 0; c)  ,
\end{align*}
as defined by \eqref{3eq: Z(c)}, \eqref{3eq: T(m1,m2;c)}, and \eqref{3eq: Ik(m1,m2;c)}. 

\subsection{Evaluation of $\boldsymbol{T (0, 0; c)}$} 

Recall from \eqref{3eq: T(m1,m2;c)} that 
\begin{align}\label{5eq: T(0,0;c)}
	T (0,0; c)  = \frac{1}{N(c)^2}\mathop{\sum\sum}_{a_1,a_2  (\mathrm{mod}\, N(c) \BZ)}
	S_{\bfF}(a_1,a_2;c)  . 
\end{align}

\begin{lem}\label{lem: T(0,0;c)}
	 We have 
	 \begin{equation}\label{eq: T00 formula}
	 	T(0,0;c)=
	 	\begin{cases}
	 	\!	\varphi(d), & \text{if } c=d\epsilon,\, d\in\BZ_+,\, \epsilon\in \mathrm{U}, 
	 		\text{and } \sqrt D d\, |  (\epsilon^2-\epsilon^{\prime \, 2});  \\
	 		0, & \text{if otherwise.}
	 	\end{cases}
	 \end{equation}
\end{lem}

\begin{proof}
	Opening the Kloosterman sum in \eqref{5eq: T(0,0;c)}, we have 
	\begin{equation*}
	T (0,0; c)  	= \frac{1}{N(c)^2} \, \sumx_{\valpha (\mathrm{mod}\, c   \SO)}   \mathop{\sum\sum}_{a_1,a_2  (\mathrm{mod}\, N(c) \BZ)} e \bigg(   \! \Tr \bigg(\frac{a_1 \valpha+ a_2 \widebar{\valpha} }{c \sqrt{D}}\bigg)  \bigg) . 
	\end{equation*}
	By orthogonality, $T (0, 0; c)$ is equal to the number of $\valpha\in(\SO/c \,\SO)^\times$ satisfying 
\begin{equation}\label{eq: T00 congruences}
	q \,| \Tr\bigg(\frac{c'\valpha}{\sqrt D}\bigg),\quad
	q \, | \Tr\bigg(\frac{c'\widebar{\valpha}}{\sqrt D}\bigg),
	\qquad q=|N(c)|.
\end{equation}
Split   $c=dc_{\ro}$ so that $(c_{\ro},c'_{\ro})=(1)$. Then  $d$ lies in either   $\BZ$ or   $\sqrt D\BZ$ (as the prime divisors of $d$ must be inert or ramified). Now the first condition in \eqref{eq: T00 congruences} implies
\begin{equation*}
	\frac{c'_{\ro}\valpha-c_{\ro}\valpha'}{\sqrt D}\equiv 0 \,  (\mathrm{mod} \, {q/d} \cdot \BZ),
\end{equation*}
and, if we reduce the above congruence modulo $c_{\ro} \SO$, then 
\begin{equation*}
	c'_{\ro}\valpha\equiv 0 \,  (\mathrm{mod} \, c_{\ro} \SO).
\end{equation*}
However, as $(c_{\ro},c'_{\ro})=(1)$ and $(\valpha,c)=(1)$, this forces $c_{\ro}$ to be a unit. Thus $T(0,0;c)=0$ unless $c=d \epsilon$ for some unit $\epsilon\in \mathrm{U}$. 

For simplicity, let us absorb $\epsilon' $ into $\valpha$ and set $\beta = \epsilon'  \valpha$. 

Next we need to rule out the case $d \in \sqrt{D} \BZ$. To this end, the first condition in \eqref{eq: T00 congruences} yields $ D \, |  (\beta + \beta' ) $.  However, we always have $ \sqrt{D} \, | (\beta - \beta') $, hence $ \sqrt{D} \,| \, 2 \beta $, and this clearly contradicts $ (\beta, d) = (1) $. 

For $d \in \BZ$,  \eqref{eq: T00 congruences} is translated into  
\begin{equation}\label{5eq: T00, 2}
	d   \,| \Tr\bigg(\frac{ \beta }{\sqrt D}\bigg),\quad
	d \, | \Tr\bigg(\frac{\epsilon^{\prime \, 2} \widebar{\beta}}{\sqrt D}\bigg). 
\end{equation}
Recall that $D$ is prime and $ D \equiv 1 (\mathrm{mod} \, 4) $. Let us write
\begin{align*}
	\beta  = a + b \frac {1+\sqrt{D}} 2, \qquad (a, b \in \BZ/ d \BZ).   
\end{align*}
Now the first condition in \eqref{5eq: T00, 2} yields $d \,|\, b$, hence $\beta \equiv a (\mathrm{mod}\, d \SO)$, so the co-primality condition $ (\beta, d) = (1) $ amounts to $ (a, d) = 1 $ and the number of these $\beta$ is exactly $\varphi (d)$.   Moreover, for such $\beta$,  it is clear that the second condition in \eqref{5eq: T00, 2} holds if and only if $ \epsilon^2 - \epsilon^{\prime\, 2} $ is divisible by $\sqrt{D} d$.    
\end{proof}

\subsection{Proof of Lemma~\ref{lem: zero-freq, asymp}: the Rational Case}
\label{subsec: proof of zero freq asymp}

First, we treat the contribution of the zero frequencies $ \EuScript{Z}_{\vkappa}(c) $ for rational $c$.  

For $ c \in \BZ \smallsetminus \{0\} $, we have proven in Lemma \ref{lem: T(0,0;c)} that $ T (0, 0; c) = \varphi (|c|) $ and we recall from \eqref{3eq: Ik(m1,m2;c)} that 
\begin{equation}\label{5eq: Ik(0,0;c)}
	I_{\vkappa} (0   , 0   ; c) = 2\pi i^{\vkappa+1} \! \int\!\!\int_{\BR_+^2} \! V_{\vkappa} (x_1, x_2)   J_{\vkappa}\bigg(\frac{4\pi\sqrt{x_1 x_2}}{|c|}\bigg)^2   \frac {\nd x_1 \nd x_2} {\sqrt{x_1 x_2}} ,
\end{equation}
where, as in \eqref{3eq: Vk(x1,x2)}, 
\begin{equation*}
	V_{\vkappa} (x_1, x_2) =	V_{\vkappa} \bigg(\frac{x_1}{\sqrt D}\bigg)V_{\vkappa}\bigg(\frac{x_2}{\sqrt D}\bigg) . 
\end{equation*}
Hence, by the definition of $\EuScript{Z}_{\vkappa}$ in \eqref{3eq: Zk},  
\begin{equation}\label{eq: Zk rational start}
	\EuScript{Z}_{\vkappa}
	=
	2 \pi i^{\vkappa+ 1} \sum_{c \, \in \BZ_+} \frac{\varphi(c)}{c^2}
	I_{\vkappa}(0,0; c). 
\end{equation}
Next, we insert the Mellin--Barnes integral representation for $J_{\vkappa} (x)^2$  as in \eqref{eq: MB Bessel square} into \eqref{5eq: Ik(0,0;c)} so that $x_1$, $x_2$, and $c$ become separate. Note that it would be convenient therein to change  $s \ra s - 1/2$.   After this, we calculate the $ x_1 $- and $x_2$-integrals by Mellin inversion, applied to the Mellin integral for $V_{\vkappa} (y) $ as in \eqref{eq: V(y)},   and evaluate the $c$-sum in \eqref{eq: Zk rational start} by the Ramanujan identity
\begin{align*}
	\sum_{c \, \in \BZ_+} \frac{\varphi(c)}{c^{1+ s}}
	=
	\frac{\zeta(  s)}{\zeta(1+ s)},
	\qquad \text{($\mathrm{Re}(s)>1$)}. 
\end{align*}
At any rate, some direct calculations yield 
\begin{equation}
	\EuScript{Z}_{\vkappa} =  \frac { 1 } { \pi i} \int_{(1)} \zeta_{\flat} (2s) \gamma_{\vkappa}^{\flat} (s - 1/2) G(s )^2  D^{s }  \frac {\nd s} {s^2} , 
\end{equation}
where 
\begin{equation}
	\zeta_{\flat} (s) = \zeta( s)\zeta(1+ s), \qquad  \gamma_{\vkappa}^{\flat} (s) = (2\pi)^{-2s} \frac{
		 \Gamma(1+2s)
		\Gamma(\vkappa-s)\Gamma(\vkappa+1+s)}
	{\Gamma (1/2)^2 \Gamma(\vkappa+1/2)^2  }. 
\end{equation}
Note that the integral may be truncated effectively at $ |\mathrm{Im} (s)| = \log k $ due to the exponential decay of $G (s)$ (see \eqref{2eq: G(s)}). By the Stirling formula in \eqref{eqapp: Stirling, quot}, for $ \mathrm{Re} (s) =  1$ and $  |\mathrm{Im} (s)| \leqslant \log k $, we have 
\begin{align*}
	 \frac{ 
	 	\Gamma(\vkappa +1/2 -s )\Gamma(\vkappa+1/2+s)}
	 {  \Gamma(\vkappa+1/2)^2  } =  1 + O \bigg(\frac{\log^2 k}{k}\bigg). 
\end{align*}
Consequently, along with the Legendre duplication formula $\sqrt{\pi} \Gamma (2s) = 2^{2s-1} \Gamma (s) \Gamma (s + 1/2) $, we obtain 
\begin{equation}
	\EuScript{Z}_{\vkappa} =   \frac { 1 } {2 \pi i} \int_{(1)} \xi (2s) \xi (1+2s) G(s )^2  D^{s }  \frac {\nd s} {s^2} + O_{D} \bigg(\frac {\log^3 k} {k} \bigg), 
\end{equation}
as desired. 

\subsection{Proof of Lemma~\ref{lem: zero-freq, bound}: the Irrational Case}
\label{subsec: proof of zero freq bound}
Now, suppose that $c\in\SO\smallsetminus\BZ$. By Lemma \ref{lem: T(0,0;c)}, we   set $ c = d \epsilon$, for $d \in \BZ_+$ and $\epsilon \in \mathrm{U} \smallsetminus \{1 \} $. For simplicity, let us assume that $ c \in \SO^{+}$ so as to avoid writing $| \ |$ everywhere. 

By symmetry, we may assume that $0 < \epsilon < 1   $. For $ c ,  c'  \Lt k^{\vepsilon}$ as in \eqref{eq: range of c}, we have $ d ,  \epsilon'  \Lt k^{\vepsilon}$, so the numbers of  such $d$ and $\epsilon$ are  $ O_{\mathbf{F}, \vepsilon} (k^{\vepsilon}) $, and $ T (0, 0; d \epsilon) = O (k^{\vepsilon}) $. 

 Recall from \eqref{3eq: Z(c)} and \eqref{3eq: IX(n;c)} that we need to consider the integral 
\begin{equation}\label{4eq: IX(0,0;c)}
	 I_{X_1, X_2}   (0   , 0   ; d \epsilon) = \frac{2\pi i^{\vkappa+1}}{\sqrt{X_1 X_2}}   \int\!\!\int  \varww_{X_1, X_2} (x_1, x_2)  NJ_{\vkappa}\bigg(\frac{4\pi\sqrt{x_1 x_2}}{d \epsilon }\bigg)   {\nd x_1 \nd x_2}, 
\end{equation}
for $  
k^{1-\vepsilon} \Lt	X_1, X_2 \Lt k^{1+\vepsilon} $ as in \eqref{3eq: range of X}. By the discussion above, it suffices to establish the following bound for $I_{X_1, X_2}   (0   , 0   ; d \epsilon)   $. 

\begin{lem}
	With the above notations and assumptions, we have 
	\begin{equation}\label{5eq: IX(0; c)}
		I _{X_1, X_2}   (0   , 0   ; d \epsilon) = O_{\mathbf{F}, \vepsilon} \bigg(\frac {k^{\vepsilon}} {\sqrt{k}} \bigg). 
	\end{equation}
\end{lem}

\begin{proof}
	 First let us set 
	 \begin{align*}
	 	x = \frac {4\pi \sqrt{x_1 x_2}} {d}, \qquad  X = \frac {4\pi \sqrt{X_1 X_2}} {d}, 
	 \end{align*}
 then the estimate in \eqref{5eq: IX(0; c)} is reduced to 
 \begin{equation}\label{5eq: IX(epsilon)}
 	I_{X} (\epsilon) = \int \varww_{X} (x)  J_{\vkappa} (x\epsilon) J_{\vkappa} (x\epsilon') \nd x \Lt  \frac {\vkappa^{\vepsilon}} {\sqrt{\vkappa}},
 \end{equation}
for $ \varww_X (x) \in C_c^{\infty} [X, 2 X]$ such that $ X ^{j}  \varww_{X }^{(j )} (x ) \Lt_{j }  1 $, and $ k^{1-\vepsilon} \Lt	X  \Lt k^{1+\vepsilon} $. 

Our   task is to prove \eqref{5eq: IX(epsilon)} by the Olver asymptotic formula (in the form of Lemma \ref{lem: Olver}) and the second derivative test (Lemma \ref{lem: 2nd derivative}). 
Put 
\begin{align*}
	\tau = \frac {\vkappa^{\vepsilon}} {\vkappa^{2/3}}. 
\end{align*}   
Note that $\epsilon \epsilon' = 1$. Write  \begin{equation*}
	y = \frac {x \epsilon} {\vkappa}, \qquad y' = \frac {x \epsilon' } {\vkappa}. 
\end{equation*}
Let us split $ I_{X} (\epsilon) $ as 
\begin{equation*}
	I_{X} (\epsilon) = I_{X}^{0} (\epsilon) + I_{X}^{\flat} (\epsilon) + I_{X}^{\infty} (\epsilon) , 
\end{equation*}
according to a suitable smooth partition of unity $ \varvv^{0} (y) + \varvv^{\flat} (y) + \varvv^{\infty} (y) \equiv 1  $ for
\begin{align*}
	(0, \infty) = (0,  1 - \tau  ] \cup [  1 -  {2} \tau ,  1 +  {2}  \tau  ] \cup [ 1 +   \tau , \infty). 
\end{align*}
By \eqref{2eq: prelim bounds} and \eqref{12eq: Bessel, smaller, K}, the contribution
$I_X^0(\epsilon)$ is negligible.  
For $ |y - 1 | \leqslant  {2} \tau $ in the transition range, since $y' = \epsilon^{\prime \, 2} y$ and $\epsilon ' > 1$, there is a  constant $ c_{\mathbf{F}} > 1 $ such that $ y' \geqslant c_{\mathbf{F}}  $. Thus, by Lemma \ref{lem: Olver} (1) and (3$^{\prime}$), 
\begin{equation*}
	J_{\vkappa}(\vkappa y)\Lt  \frac 1 {\vkappa^{ 1/3}},\qquad
	J_{\vkappa}(\vkappa y')\Lt_{\bfF} \frac 1 {\vkappa^{1/2}}.
\end{equation*}
As the support of $ \varvv^{\flat} (x \epsilon / \vkappa) $  has length
$O(\vkappa\tau \epsilon' ) = O (\vkappa^{1/3+\vepsilon}  )$, we obtain the bound
\begin{align*}
	 I_X^{\flat}(\epsilon)
	 \Lt_{\bfF,\vepsilon} \frac {\vkappa^{\vepsilon}} {\sqrt{\vkappa}}.
\end{align*}
Therefore, it remains to analyze $ I_{X}^{\infty} (\epsilon) $. Keep in mind that  we still have $y' \geqslant c_{\mathbf{F}} > 1$ on the support of $ \varww_{  X} (x) \varvv^{\infty} (x \epsilon / \vkappa) $. Next, we insert the asymptotic formula \eqref{Beq: Olver} in Lemma \ref{lem: Olver} (3) into the integral $I_{X}^{\infty} (\epsilon)$ (see \eqref{5eq: IX(epsilon)}).  It is clear that the error terms contribute  $ O (\vkappa^{\vepsilon} / \sqrt{\vkappa }) $.  For the four main-term integrals,  let us further apply a smooth dyadic partition to $\varvv^{\infty} (y) $. Thus,  for $\tau \Lt \delta \Lt \vkappa^{\vepsilon}$ and   $ \vkappa  \Lt	X  \Lt \vkappa^{1+\vepsilon} $, we need to estimate 
\begin{equation}\label{5eq: Idelta(epsilon)}
	I_{\delta, X}^{_{\pm \pm}} (\epsilon) = \frac 2 {\vkappa}  \int  \varww_{\delta, X}^{_{\pm \pm}} (x; \epsilon) \exp (i f^{\phantom{1}}_{^{\pm \pm}} (x; \epsilon) ) \nd x , 
\end{equation} 
with the weight function
\begin{align*}
	 \varww_{\delta, X}^{_{\pm \pm}} (x; \epsilon)   = \frac {\varww_{  X} (x) \varvv_{\delta} ( y  ) W_{\pm} (\vkappa \gamma (y)) W_{\pm} (\vkappa \gamma (y'))  } { (y^2-1)^{1/4} (y'^{\,2} -1)^{1/4} }   , 
	  %=  \frac {\varww_{  X} (x) \varvv_{\delta} ( {x \epsilon} / {\vkappa}  )  } { (({x \epsilon} /{\vkappa})^2-1)^{1/4} ( ( {x \epsilon' } /{\vkappa})^{ 2} -1)^{1/4} }   , 
\end{align*}
and the phase function
\begin{equation*}
	f^{\phantom{1}}_{^{\pm \pm}} (x; \epsilon) =   \pm {\vkappa} \gamma (y)
	\pm {\vkappa} \gamma (y') = \pm {\vkappa} \gamma \Big(   \frac{x\epsilon}{\vkappa}   \Big)
	\pm {\vkappa} \gamma\Big(\frac{x\epsilon'}{\vkappa}\Big), 
\end{equation*}
where $ \varvv_{\delta} (y) \in C_c^{\infty} [1 + \delta, 1 + 2 \delta]  $ and $  \varww_{X} (x) \in C_c^{\infty} [X, 2 X] $ satisfy $ \delta ^{j}  \varvv_{\delta }^{(j )} ( y ) \Lt_{j }  1 $ and $ X ^{j}  \varww_{X }^{(j )} (x ) \Lt_{j }  1 $. 
Recall from \eqref{2eq: gamma} that 
\begin{equation*}
	\gamma (y) = \sqrt{y^2-1} - \mathrm{arcsec}\, y ,
\end{equation*}
so that 
\begin{equation*}
	\gamma' (y) = \frac{\sqrt{y^2 - 1}}{y} , \qquad   \gamma'' (y) = \frac{1}{y^2 \sqrt{y^2 - 1}} . 
\end{equation*}
Consequently, 
\begin{equation*}
	\frac {\partial^2 f^{\phantom{1}}_{^{\pm \pm}} (x; \epsilon)} {\partial x^2} 
	=\frac{\vkappa^2}{ x^2}\bigg( \!
	\pm\frac{1}{\sqrt{x^2 \epsilon^2-\vkappa^2}}
	\pm\frac{1}{\sqrt{ x^2\epsilon'^{\, 2} -\vkappa^2 }}
	\bigg).
\end{equation*}
Since 
\begin{align*}
	 \frac {\sqrt{ x^2\epsilon'^{\, 2} -\vkappa^2 }} {\sqrt{x^2 \epsilon^2-\vkappa^2}} > \frac {\epsilon'} {\epsilon} = \epsilon'^{\,2} > c_{\mathbf{F}} > 1, 
\end{align*}
we obtain the lower bound
\begin{equation}\label{5eq: f''}
	\bigg| \frac {\partial^2 f^{\phantom{1}}_{^{\pm \pm}} (x; \epsilon)} {\partial x^2}  \bigg| \Gt_{\mathbf{F}} \frac {\vkappa^2 }   {x^2 \sqrt{x^2 \epsilon^2-\vkappa^2}} \Gt  \frac{\vkappa}{X^2 \sqrt{\delta( \delta + 1)} }, 
\end{equation}
on the support of $\varww_{  X} (x) \varvv_{\delta} (x\epsilon / \vkappa)$. 
Moreover, it is routine to verify that the variation 
\begin{equation}\label{5eq: variation}
	\textit{Var} \,(\varww_{\delta, X}^{_{\pm \pm}} (\, \cdot \, ; \epsilon)) \Lt \frac{1}{  \delta^{1/4}( \delta + 1)^{3/4} } . 
\end{equation}
Finally, in view of \eqref{5eq: Idelta(epsilon)}--\eqref{5eq: variation}, it follows from Lemma \ref{lem: 2nd derivative} that 
\begin{equation*}
	 I_{\delta, X}^{_{\pm \pm}} (\epsilon) \Lt   \frac 1 {\vkappa} \cdot \frac{X}{\sqrt{\vkappa    (\delta+1)}  }  \Lt \frac {\vkappa^{\vepsilon}} {\sqrt{\vkappa}} , 
\end{equation*}
as desired.  
\end{proof}

\section{Estimate for the Dual Sum: Proof of Lemma \ref{lem: dual}} 

In view of  \eqref{3eq: IX (n;r;c)}--\eqref{3eq: dual Q(r;c)},  we need to consider the sum 
\begin{align*}
	\EuScript{Q}^{\flat}_{X_1, X_2}(r;c) = \mathop{\mathop{\sum \sum}_{ n_1   , n_2     \in \BZ  }}_{(n_1   , n_2   ) \neq (0, 0)}  
	\big| I_{X_1, X_2} (n_1, n_2; r;  |N(c)| ) \big|, 
\end{align*} in which the integral 
\begin{equation*}
I_{X_1, X_2} (n_1, n_2; r;  q) =	\frac{1}{\sqrt{X_1 X_2}}\iint
	\varww_{X_1, X_2} (x_1,x_2)e\bigg(\frac{x_1x_2/r-n_1 x_1-n_2 x_2}{q} \bigg)\nd x_1\nd x_2, 
\end{equation*}
for 
\begin{equation} \label{6eq: ranges}
		k^{1-\vepsilon} \Lt	X_1, X_2 \Lt k^{1+\vepsilon} , \qquad q \Lt k^{\vepsilon}, \qquad r \Gt k  . 
\end{equation}  
By the changes $x_1 \ra X_1 x_1$ and $ x_2 \ra X_2 x_2 $, we rewrite the integral as 
\begin{eqnarray*}
	\sqrt{X_1 X_2} \iint
	\varww  (x_1,x_2)e\bigg(\frac{X_1  X_2 x_1x_2/r- X_1 n_1 x_1- X_2 n_2 x_2}{q} \bigg)\nd x_1\nd x_2,
\end{eqnarray*}
where $\varww (x_1, x_2) \in C_c^{\infty} ([1, 2]^2)$ has bounds $ \varww^{(j_1, j_2)} (x_1, x_2) \Lt_{j_1, j_2} 1  $. 

First, we consider the case when exactly one dual variable is equal to $0$. Suppose, for example, that $n_1\neq0$ and $n_2=0$. It would be convenient to set $ v = x_1 x_2 $ as the new variable so that the phase function is turned into
\begin{align*}
	 \frac{X_1  X_2 v /r - X_1 n_1 x_1 }{q}.  
\end{align*}
Now its partial derivative with respect to $ x_1 $ is  $ - X_1 n_1  / q$. As $ X_1 / q \Gt k^{1-\vepsilon}$, it follows from Lemma \ref{lem: stationary phase, dim 1} that the integral is negligibly small (and of rapid decay in $n_1$). 

Next, we consider the case when $n_1 n_2 \neq 0$. As the $x_1$- and $x_2$-partial derivatives of the phase function read
\begin{align*}
	 \frac {X_1 X_2 x_2 } {q r} - \frac {X_1 n_1} {q} , \qquad \frac {X_1 X_2 x_1 } {q r} - \frac {X_2 n_2} {q}, 
\end{align*}
respectively, by applying Lemma \ref{lem: stationary phase, dim 1}, we infer that  the $x_1$- or $x_2$-integral is negligibly small unless 
\begin{equation}\label{eq: range of n1, n2}
	\bigg|\frac{ X_2 x_2 }{r} - n_1 \bigg|
	\Lt \frac {q k^{\vepsilon}} {X_1}, \qquad  	\bigg|\frac{ X_1 x_1 }{r} - n_2 \bigg|
	\Lt \frac {q k^{\vepsilon}} {X_2}. 
\end{equation}
As $ X_2 / r, X_1 / r \Lt k^{\vepsilon} $ and $q/X_1, q/X_2 \Lt k^{\vepsilon} / k$ by \eqref{6eq: ranges}, we may restrict the dual $n_1$- and $n_2$-sums to the range $ 0 <  n_1,  n_2  \Lt k^{\vepsilon} $. Further, the range $ r \Gt k $ as in \eqref{6eq: ranges}   may be strengthened  into 
\begin{equation}
	\label{6eq: range of r} 
	\vkappa \Lt r < \vkappa^{1+\vepsilon}. 
\end{equation}  Now the $x_2$-interval defined by the first inequality in \eqref{eq: range of n1, n2} has length $ O (r q k^{\vepsilon}/ X_1 X_2) = O (k^{1+\vepsilon}/X_1 X_2) $, so 
	\begin{equation}
		I_{X_1, X_2} (n_1, n_2; r;  q) \Lt \sqrt{X_1 X_2} \frac {k^{1+\vepsilon}} {X_1 X_2} \Lt   {k^{\vepsilon}}  .  
	\end{equation}

At any rate, $\EuScript{Q}^{\flat}_{X_1, X_2}(r;c) $ is negligibly small unless $ k \Lt r \Lt k^{1+\vepsilon} $ by  \eqref{6eq: range of r}, and for such $r$ we conclude that 
\begin{equation}\label{6eq: bound for Q(r;c)}
	\EuScript{Q}^{\flat}_{X_1, X_2}(r;c) \Lt k^{\vepsilon}. 
\end{equation} 
Therefore, up to a negligible error, it follows from \eqref{3eq: dual O(c)} and \eqref{6eq: bound for Q(r;c)} 
that 
\begin{equation*}
	\widetilde{\EuScript{O}}_{X_1, X_2}  (c) \Lt \vkappa^{\vepsilon} \! \int_{\vkappa/8\pi}^{\vkappa^{1+\vepsilon}} | J_{\vkappa}(4\pi r) |    \frac{\nd r}{r} .
\end{equation*}
Finally, by Cauchy--Schwarz and \eqref{2eq: integral of J(x)2}, 
\begin{align*}
	 \widetilde{\EuScript{O}}_{X_1, X_2}  (c)^2  \Lt \vkappa^{\vepsilon} \! \int_{\vkappa/8\pi}^{\vkappa^{1+\vepsilon}}  J_{\vkappa}(4\pi r) ^2  \frac{\nd r}{r} < \vkappa^{\vepsilon} \! \int_{0}^{\infty}   J_{\vkappa}(4\pi r) ^2  \frac{\nd r}{r} = \frac {\vkappa^{\vepsilon}} {2 \vkappa}, 
\end{align*}
as desired.

\section{Proof of Theorem \ref{thm: 1st moment}} 

More generally, for $ m, n \in \BZ_+$, define 
\begin{equation}
\EuScript{M}_k^{0} (m, n) = 
\sum_{f\in H_k}\omega_f\lambda_f(m) \lambda_f(n) , 
\end{equation}
\begin{equation}
	\EuScript{M}_k^{1} (m) = 
	\sum_{f\in H_k}\omega_f\lambda_f(m) L(1/2,\As(f)).  
\end{equation}

\begin{proposition}
Let $A > 0$.  	We have 
	\begin{equation}\label{7eq: M0k(m)}
		 \EuScript{M}_k^{0} (m, n) = { (k-1)^2} D \sqrt{  D  } \cdot \delta (m, n) +
		 O_{\bfF,\vepsilon, A}(k^{-A}), 
	\end{equation}
for $ m n  \Lt k^{2- \vepsilon}$, where $\delta (m, n)$ is the Kronecker $\delta$ that detects $m = n$,  and 
\begin{equation}\label{7eq: M1k(m)} 
	\EuScript{M}_k^{1} (m) = {  (k-1)^2} D \sqrt{\frac D m} 
	\bigg(
	\psi\bigg(k-\frac12\bigg)
	+\gamma
	+\log\bigg(\frac{\sqrt D}{16\pi^2 m}\bigg)
	\bigg)
	+
	O_{\bfF,\vepsilon, A}(k^{-A}), 
\end{equation}
for $ m \Lt k^{1-\vepsilon}$, where  $\psi (s)  $ is  the di-gamma function and $\gamma$ is the Euler constant.  
\end{proposition}

\begin{proof}
	It is clear that  \eqref{7eq: M0k(m)} follows directly from the Petersson trace formula in \eqref{eq: Petersson} and the bound for the Bessel function in \eqref{2eq: prelim bounds}. Similarly, 
	\begin{align*}
		 \EuScript{M}_k^{1} (m) = { 2 (k-1)^2} D \sqrt{  \frac D m } V_{k-1} \bigg(\frac{m}{\sqrt D}\bigg) +  O_{\bfF,\vepsilon, A }(k^{-A}), 
	\end{align*} 
which is a direct consequence of \eqref{eq: AFE}, \eqref{2eq: bound for Vk(y)}, \eqref{eq: Petersson}, and \eqref{2eq: prelim bounds}. Recall from \eqref
{eq: V(y)} that 
\begin{equation*}
	V_{\vkappa} \bigg(\frac{m}{\sqrt D}\bigg) =\frac{1}{2\pi i}\int_{(3)}  \zeta(1+2 v)
	\frac{\gamma_{\vkappa} (1/2+ v )}{\gamma_{\vkappa} (1/2 )}  G (v)     \big({\sqrt{D}} / {m} \big) ^{v} \frac{\nd v}{v} . 
\end{equation*}
Thus \eqref{7eq: M1k(m)}  follows from shifting the integral contour to the far left, say $\mathrm{Re} (v) = - 2/\vepsilon - A /\vepsilon$, and calculating the residue at the double pole $v  = 0$; here we need to use the definition of $ \gamma_{\vkappa} (s) $ and the expansion of $\zeta (1+s)$ as in \eqref{2eq: gamma k(s)} and \eqref{eq: Stieltjes constants definition}.    
\end{proof}

Theorem \ref{thm: 1st moment} is a direct consequence of \eqref{7eq: M0k(m)} and \eqref{7eq: M1k(m)}, along with the Stirling formula (see \eqref{app1: Stirling, 2})
\begin{align*}
		  \psi \bigg(k-\frac12\bigg)  = \log k   + O  \bigg( \frac 1 { k } \bigg) . 
\end{align*}

\appendix

\section{Stirling's Formulae} 

According to \cite[\S\S 1.1, 1.2]{MO-Formulas}, for   $ |\arg (s) | < \pi   $, as $|s| \ra \infty$, we have 
\begin{equation}\label{eqapp: Stirling, 1}
	\log \Gamma (s) =    \bigg( s - \frac 1 2 \bigg)  \log s - s   + \frac 1 2 \log (2\pi) + O  \bigg( \frac 1 {|s|} \bigg)  .  % \frac 1 {12 s} + O  \bigg( \frac 1 {|s|^{3}}  \bigg). 
\end{equation}
\begin{equation}\label{app1: Stirling, 2} 
	\psi (s)   = \log s   + O \lp \frac 1 {|s| } \rp,  
\end{equation}
\delete{\begin{align}%\label{6eq: Stirling, 1}
	\begin{aligned}
		\log \Gamma (s) =  &  \bigg( s - \frac 1 2 \bigg)  \log s - s   + \frac 1 2 \log (2\pi) +   \sum_{m=1}^{M} %\frac {B_{2m}} {2m(2m-1)} 
		 \frac {A_{2m}} {s^{2m-1}}  
		 +   O   \bigg( \frac 1 {|s|^{2M+1}}  \bigg) ,
	\end{aligned}
\end{align} 
where the coefficients $A_{2m}$ and the Bernoulli numbers $B_{2m}$ are related via
\begin{align*}
A_{2m} =	\frac {B_{2m}} {2m(2m-1)} . 
\end{align*} }
Let $ |\arg (s) | < \pi$ and $|s|$ be large. From \eqref{eqapp: Stirling, 1} we deduce the asymptotic formula
\begin{equation}\label{eqapp: Stirling, quot}
	\frac {\Gamma (s+\valpha)} {\Gamma (s)} =  s^{\valpha} \bigg(1 + O  \bigg(\frac { 1 + |\valpha|^2} {|s|} \bigg) \bigg),  % \frac {\valpha^2 - \valpha} {2s} + O \bigg(\frac {|\valpha|^4 + |\valpha|} {|s|^2} \bigg) \bigg), 
\end{equation}
for any $|\valpha| \Lt \sqrt{|s|}$.

	\section{Stationary Phase}
Finally, we record here Lemmas A.1 and A.2 in \cite{AHLQ-Bessel}, as variants of  \cite[Lemma {\rm 8.1}]{BKY-Mass} and \cite[Lemma 5.1.3]{Huxley}.   

\begin{lem}\label{lem: stationary phase, dim 1}
	Let $\varww   \in C_c^{\infty} (a, b)$. Let  $f  \in C^{\infty} [a, b]$ be real-valued.  Suppose that there
	are parameters $P,\, Q,\, R,\, S,\, Z  > 0$ such that
	\begin{align*}
		f^{(i)} (x) \Lt_{ \, i } Z / Q^{i}, \qquad \varww^{(j)} (x) \Lt_{ \, j } S / P^{j},
	\end{align*}
	for  $i \geqslant 2$ and $j \geqslant 0$, and
	\begin{align*}
		| f' (x) | \Gt R. 
	\end{align*}
	Then 
	\begin{align*}
		\int_a^b  e (f(x)) \varww (x)  \nd x \Lt_{A} (b - a) S \bigg( \frac {Z} {R^2Q^2} + \frac 1 {R Q} + \frac 1 {R P} \bigg)^A  
	\end{align*} 
	for any  $A > 0$.
\end{lem}

\begin{lem}\label{lem: 2nd derivative}
	Let $\varww   \in C_c^{\infty} [a, b]$ and $V$ be its total variation. 	Let $f \in C^{\infty} [a, b]$ be real-valued. If $|f'' (x)| \geqslant \lambda > 0$ on $[a, b]$, then 
	\begin{align*}
		\bigg|\int_a^b  e (f(x)) \varww (x)  \nd x  \bigg| \leqslant \frac {4 V} {\sqrt{\pi \lambda}} . 
	\end{align*}
\end{lem}
	
	%\bibliographystyle{alphanum}
%	\bibliography{references}

\begin{thebibliography}{AHLQ}
	
	\bibitem[AHLQ]{AHLQ-Bessel}
	K.~Aggarwal, R.~Holowinsky, Y.~Lin, and Z.~Qi.
	\newblock A {B}essel delta method and exponential sums for {${\rm GL}(2)$}.
	\newblock {\em Q. J. Math.}, 71(3):1143--1168, 2020.
	
	\bibitem[Asa]{Asai-1977}
	T.~Asai.
	\newblock On certain {D}irichlet series associated with {H}ilbert modular forms
	and {R}ankin's method.
	\newblock {\em Math. Ann.}, 226(1):81--94, 1977.
	
	\bibitem[BKY]{BKY-Mass}
	V.~Blomer, R.~Khan, and M.~P. Young.
	\newblock Distribution of mass of holomorphic cusp forms.
	\newblock {\em Duke Math. J.}, 162(14):2609--2644, 2013.
	
	\bibitem[CFK{\etalchar{+}}]{Conrey-FKRS-Moments}
	J.~B. Conrey, D.~W. Farmer, J.~P. Keating, M.~O. Rubinstein, and N.~C. Snaith.
	\newblock Integral moments of {$L$}-functions.
	\newblock {\em Proc. London Math. Soc. (3)}, 91(1):33--104, 2005.
	
	\bibitem[Has]{Hasse-NT}
	H.~Hasse.
	\newblock {\em Number {T}heory}. {Grundlehren der
		Mathematischen Wissenschaften}, Vol.~229.
	\newblock Springer-Verlag, Berlin-New York, 1980.
	
	\bibitem[Hux]{Huxley}
	M.~N. Huxley.
	\newblock {\em Area, {L}attice {P}oints, and {E}xponential {S}ums}.  {London Mathematical Society Monographs, New Series}, Vol.~13.
	\newblock The Clarendon Press, Oxford University Press, New York, 1996.
	\newblock Oxford Science Publications.
	
	\bibitem[IK]{IK}
	H.~Iwaniec and E.~Kowalski.
	\newblock {\em Analytic {N}umber {T}heory}.  {  American
		Mathematical Society Colloquium Publications}, Vol.~53.
	\newblock American Mathematical Society, Providence, RI, 2004.
	
	\bibitem[Kha]{Khan-Sym2-Non-vanishing}
	R.~Khan.
	\newblock Non-vanishing of the symmetric square {$L$}-function at the central
	point.
	\newblock {\em Proc. Lond. Math. Soc. (3)}, 100(3):736--762, 2010.
	
	\bibitem[KY]{Khan-Young-Sym2}
	R.~Khan and M.~P. Young.
	\newblock Moments and hybrid subconvexity for symmetric-square {$L$}-functions.
	\newblock {\em J. Inst. Math. Jussieu}, 22(5):2029--2073, 2023.
	
	\bibitem[LQ]{LQi-Asai-LS}
	C.~Li and Z.~Qi.
	\newblock On the second moment of $l (1/2, \textrm{{A}s}(f)\times \phi)$.
	\newblock 2026.
	
	\bibitem[Luo1]{Luo-2003-Hilbert}
	W.~Luo.
	\newblock Poincar\'e{} series and {H}ilbert modular forms.
	\newblock {\em Ramanujan J.}, 7(1-3):129--140, 2003.
	\newblock Rankin memorial issues.
	
	\bibitem[Luo2]{Luo-Asai}
	W.~Luo.
	\newblock Moments of the central {$L$}-values of the {A}sai lifts.
	\newblock {\em Canad. Math. Bull.}, 67(3):796--804, 2024.
	
	\bibitem[MOS]{MO-Formulas}
	W.~Magnus, F.~Oberhettinger, and R.~P. Soni.
	\newblock {\em Formulas and {T}heorems for the {S}pecial {F}unctions of
		{M}athematical {P}hysics}.
	\newblock 3rd enlarged edition. Die Grundlehren der mathematischen
	Wissenschaften, Band 52. Springer-Verlag New York, Inc., New York, 1966.
	
	\bibitem[OLBC]{NIST}
	Frank W.~J. Olver, Daniel~W. Lozier, Ronald~F. Boisvert, and Charles~W. Clark.
	\newblock {\em N{IST} {H}andbook of {M}athematical {F}unctions}.
	\newblock U.S. Department of Commerce, National Institute of Standards and
	Technology, Washington, DC; Cambridge University Press, Cambridge, 2010.
	
	\bibitem[Olv1]{Olver-1}
	F.~W.~J. Olver.
	\newblock The asymptotic solution of linear differential equations of the
	second order for large values of a parameter.
	\newblock {\em Philos. Trans. Roy. Soc. London. Ser. A.}, 247:307--327, 1954.
	
	\bibitem[Olv2]{Olver-Bessel}
	F.~W.~J. Olver.
	\newblock The asymptotic expansion of {B}essel functions of large order.
	\newblock {\em Philos. Trans. Roy. Soc. London. Ser. A.}, 247:328--368, 1954.
	
	
	
	\bibitem[Qi]{Qi-GL(3)}
	Z.~Qi.
	\newblock Subconvexity for {$L$}-functions on {$\rm GL_3$} over number fields.
	\newblock {\em J. Eur. Math. Soc. (JEMS)}, 26(3):1113--1192, 2024.
	
	\bibitem[Wat]{Watson}
	G.~N. Watson.
	\newblock {\em A {T}reatise on the {T}heory of {B}essel {F}unctions}.
	\newblock Cambridge University Press, Cambridge, England; The Macmillan
	Company, New York, 1944.
	
\end{thebibliography}

\newcommand{\etalchar}[1]{$^{#1}$}

\def\cprime{$'$}

\end{document}